\documentclass[11pt,reqno]{amsart}
\usepackage{fullpage}

\usepackage{amsfonts,amssymb,amsmath}
\usepackage{mathrsfs}
\usepackage{dsfont}
\usepackage{mathtools}
\usepackage{amsthm}
\usepackage{mathtools}
\usepackage{leftindex,tensor}
\usepackage[OT1]{fontenc}
\usepackage{MnSymbol}
\usepackage{manfnt}
\usepackage[all,cmtip]{xy}
\usepackage{braket}

\usepackage{tikz-cd}
\tikzset{%
   symbol/.style={%
        draw=none,
        every to/.append style={%
            edge node={node [sloped, allow upside down, auto=false]{$#1$}}}
    }
}

\usepackage[pagebackref,hyperindex,linktocpage=true]{hyperref}
\hypersetup{
    colorlinks,
    linkcolor={red!50!black},
    citecolor={blue!50!black},
    urlcolor={blue!80!black}
}

\usepackage{color}

\newtheorem{theo}{Theorem}[section]
\newtheorem{theo-def}{Theorem-Definition}[section]
\newtheorem{lem}[theo]{Lemma}
\newtheorem{lem-def}[theo]{Lemma-Definition}
\newtheorem{prop}[theo]{Proposition}
\newtheorem{cor}[theo]{Corollary}

\theoremstyle{definition}
\newtheorem{defi}[theo]{Definition}
\newtheorem{ex}[theo]{Example}
\newtheorem{rqe}[theo]{Remark}

\numberwithin{equation}{section}

\def\L{\mathcal{L}}

\def\E{\mathcal{E}}
\def\F{\mathcal{F}}

\def\L{\mathcal{L}}
\def\M{\mathcal{M}}
\def\O{\mathcal{O}}
\def\A{\mathcal{A}}

\def\ZZ{\mathbb{Z}}

\def\RR{\mathbb{R}}

\def\GG{\mathbb{G}}
\def\QQ{\mathbb{Q}}

\def\bfO{\mathbf{O}}

\def\ad{\mathrm{ad}}
\def\Ad{\mathrm{Ad}}

\def\la{\lambda}

\def\enh{\mathrm{enh}}
\def\Vin{\mathrm{Vin}}

\DeclareMathOperator{\Aut}{Aut}

\DeclareMathOperator{\GL}{GL}

\DeclareMathOperator{\Mon}{Mon}

\DeclareMathOperator{\Out}{Out}
\DeclareMathOperator{\Spec}{Spec}

\DeclareMathOperator{\Hom}{Hom}

\DeclareMathOperator{\Cone}{Cone}
\DeclareMathOperator{\Dyn}{Dyn}

\DeclareMathOperator{\Kil}{Kil}
\DeclareMathOperator{\Lie}{Lie}
\DeclareMathOperator{\RD}{RD}

\def\SheafCone{\mathop{\mathcal{C}\!\mathit{one}}\nolimits}
\def\SheafHom{\mathop{\mathcal{H}\!\mathit{om}}\nolimits}
\def\SheafIsom{\mathop{\mathcal{I}\!\mathit{som}}\nolimits}

\newcommand{\epic}{\twoheadrightarrow}
\newcommand{\into}{\hookrightarrow}

\title{Classification of reductive monoids}
\author{Jingren Chi}
\email{jrenchi@amss.ac.cn}
\address{Morningside Center of Mathematics and State Key Laboratory of Mathematical Sciences, Academy of Mathematics and Systems Science, Chinese
Academy of Sciences, Beijing 100190, China;}

\begin{document}
\begin{abstract}
    We classify reductive monoids over arbitrary base schemes by descent from the split classification in \cite{CJ}. For a reductive group scheme $G$, we construct its sheaf of based root data, abstract Cartan torus, abstract Weyl group, and a scheme of Weyl cones. The last of these represents the functor of isomorphism classes of rigidified reductive monoids with unit group $G$. As applications, we deduce a criterion for the existence of monoids that are not groups, prove an extension theorem over normal integral bases, and construct the Vinberg monoids without a quasi-split hypothesis.
\end{abstract}
\maketitle

\section{Introduction}
The goal of this article is to classify reductive monoids  over arbitrary base schemes in complete generality. In \cite{CJ} we have solved this problem when the unit group is split and the base scheme is connected. We build on this result and perform the reduction step here. In particular, this generalizes the works of Vinberg \cite{Vinberg1995}, Rittatore \cite{RittatoreMonoidArticle} and Renner \cite{Renner} over algebraically closed fields. We refer the reader to the introduction in \cite{CJ} for further discussion on the historical backgrounds. 

\subsection{Main results}
In \cite{CJ} split reductive monoids over connected schemes are classified in terms of weight monoids (see [\emph{loc. cit.}, Definition 2.11]) in the dominant weight cone. In this paper we will use certain combinatorial object called \emph{the Weyl cones}, which carry essentially the same information as the weight monoids. These are by definition the Weyl group stable rational convex polyhedral cones of full dimension in the $\QQ$-vector space generated by the weight lattice. The weight monoids are used in \emph{loc. cit.} mainly because they are more relevant to the explicit construction of the split reductive monoids. To perform the descent step and obtain a general classification theorem, we find it better to use the Weyl cones because they do not depend on choices of global pinnings (which are not available for non-quasi-split groups). The rough idea for the descent step is as follows. Given a reductive group scheme $G$ over a base scheme $S$. If we view it informally as a family of split reductive groups on $S$, we should use a family of Weyl cones over $S$ to classify the reductive monoids over $S$ with unit group $G$. This is formalized by introducing a sheaf on $S$ whose fiber at a geometric point $\Bar{s}\in S$ is the set of Weyl cones for the group $G_{\Bar{s}}$. We show that this sheaf is represented by a scheme that is \'etale locally constant over $S$. To achieve this, we sheafify the usual notion of based root data, and define the abstract Cartan and abstract Weyl group for $G$ in a canonical way, without any choices. Our main result is Theorem \ref{thm:general-classification}, which identifies rigidified reductive $G$-monoids with global sections of the scheme of Weyl cones. Forgetting the rigidification corresponds to taking orbits under the global section of the automorphism group $\Aut_G(S)$.\par 
We will derive the following applications of our main result. In Corollary \ref{cor:genuine-monoid} we show that there exists a reductive monoid over $S$ with unit group $G$ but not itself equal to $G$ if and only if the center admits a nonzero $S$-group homomorphism $\GG_m\to Z_G$. In Corollary \ref{cor:integral-normal-base} we show that over a normal integral base scheme, every rigidified reductive monoid over the generic point extends uniquely (up to rigidified isomorphisms) once its unit group extends to a reductive group scheme. In \S\ref{sec:Vinberg-monoid}, we construct the Vinberg monoid for any reductive group scheme, quasi-split or not, generalizing the case of split groups in \cite[\S5.2,\S5.3]{CJ}.
\subsection{Plan of the article}
Section \S\ref{sec:review-red-group} recalls the needed facts about reductive group schemes. Section \S\ref{sec:split-monoid} reformulates the split classification over connected bases \cite{CJ} in terms of Weyl cones. In Section \S\ref{sec:RD-Cone-scheme} we construct the canonical sheaves associated with a reductive group scheme, extend the split classification to disconnected bases, and prove independence of the pinning. Section \S\ref{sec:general-classification} first treats the case of a field and then proves the general classification theorem and its applications. Section \S\ref{sec:example} discusses basic examples of the multiplicative monoids of central simple algebras and the Vinberg monoids.
\subsection{General notations and conventions}
All categories are assumed to be small. A \emph{thin category} is the same as a \emph{preordered set}, and is canonically equivalent to a poset (partially ordered set). For an abelian group $X$ we put $X_\QQ\coloneqq X\otimes_\ZZ\QQ$. For any set $E$ and any scheme $S$, we let $\underline{E}_S$ denote the constant scheme over $S$ with fiber $E$, i.e. ``$E$ copies of $S$". \par 
For a group scheme $H$ over a scheme $S$, we let $\Lie(H)$ denote its Lie algebra, which is a sheaf of $\O_S$-modules on $S$. We let $X^*(H)\coloneqq\Hom_{S\text{-gp}}(H,\GG_m)$ denote the global character group of $H$ and let $X_*(H)\coloneqq\Hom_{S\text{-gp}}(\GG_m,H)$ be the group of global cocharacters. For any sub-torus $T\subset H$ and any $\alpha\in X^*(T)$, we let $\Lie(H)_\alpha\subset\Lie(H)$ denote the largest sub-$\O_S$-module on which the adjoint $T$-action factors through the character $\alpha$.  

\subsection{Related work}
The statement of our main result Theorem \ref{thm:general-classification} can be viewed as a direct generalization of Renner's classification of reductive monoids over algebraically closed fields (see \cite[Theorem 5.4]{Renner}). Our previous work \cite{CJ} that establish the special case of split monoids over connected base generalizes Vinberg's approach \cite{Vinberg1995} to the classification problem over algebraically closed field of characteristic $0$. As a comparison, the ``dual" approach of Rittatore \cite{RittatoreMonoidArticle}, classifies reductive monoids over any algebraically closed fields in terms of cones in the coweight space (as opposed to the weight space in our approach), builds on the theory of spherical embeddings of Luna-Vust \cite{Luna-Vust} and Knop \cite{Knop-LunaVust}. This latter theory has been generalized by Huruguen \cite{Huruguen} and Wedhorn \cite{Wedhorn} to arbitrary fields. It would be an interesting question whether one can have a classification result for spherical embeddings over general base schemes. 

\subsection*{Acknowledgment}
I am grateful to Simon Jacques for collaborations on previous projects that led to the current work. I thank Kazuhiro Ito for helpful conversations. The work leading to the initial draft of this paper was done by the author (a human). During the revision process ChatGPT 6 Astra was used for proofreading and text improvements. Most notably, the AI tool streamlined the formulation and proof of Lemma \ref{lem:Weyl-cone}, suggested the terminology in Definition \ref{def:canonical-generator-Weyl-cone} and found an imprecise point in the literature (see the proof of Lemma \ref{lem:pinning}). The author has been supported by National Key R\&D Program of China No.2023YFA1009701, National Natural Science Foundation of China (Grant No.12288201, No.12231001), CAS Project for Young Scientists in Basic Research (Grant No.YSBR-033).

\section{Review of reductive group schemes}\label{sec:review-red-group}
In this section, we review some basic concepts about reductive group schemes. All unspecified citations refer to \cite{SGA3-3new}.
\subsection{Root data}
A \emph{root datum} is a quadruple $(X,X^\vee,R,R^\vee)$ where $X$ (resp. $X^\vee\coloneqq\Hom_\ZZ(X,\ZZ)$) is the weight (resp. coweight) lattice (both are finite free $\ZZ$-modules), $R\subset X$ (resp. $R^\vee\subset X^\vee$) is the finite set of roots (resp. coroots) with a bijection $R\xrightarrow{\sim}R^\vee$ satisfying the axioms in [XXI, 1.1.1]. 
Let $\RD$ be the set of isomorphism classes of \emph{reduced} root data.\par
A \emph{based root datum} consists of a root datum $(X,X^\vee,R,R^\vee)$ together with a choice of a set of simple roots $\Delta\subset R$ and a set of simple coroots $\Delta^\vee\subset R^\vee$ that are in bijection under the root-coroot correspondence $R\xrightarrow{\sim}R^\vee$. Since the simple roots $\Delta$ and coroots $\Delta^\vee$, together with their pairing, recover the root datum, we will denote a based root datum simply by $\Psi=(X,X^\vee,\Delta,\Delta^\vee)$. 

\subsection{Splittings, types and Killing pairs}\label{sec:group-splitting}
Let $S$ be a nonempty base scheme. Following [XIX, 2.7], a \emph{reductive group scheme} over $S$ is an affine smooth group scheme over $S$ with connected reductive geometric fibers.\par 
Let $G$ be a reductive group scheme over $S$. We recall the following objects associated to $G$ (which may or may not exist).
\subsubsection{Splittings} A \emph{splitting} for $G$ consists of the data $(T,X,R)$ where $T\subset G$ is a maximal torus (over $S$), $X$ is a free abelian group of finite rank equipped with an isomorphism of $S$-group schemes $\Spec(\O_S[X])\cong T$, where $\O_S[X]$ is the sheaf of group algebras of $X$ with coefficients in $\O_S$, and $R\subset X$ is a root system of $T$ in $G$ such that the root line bundle $\mathrm{Lie}(G)_\alpha$ is trivial for all $\alpha\in R$, see [XXII, 1.13]. Any splitting gives rise to a reduced root datum $(X,X^\vee,R,R^\vee)$. If a splitting exists, the group $G$ is said to be split (over $S$). From [XXII, 2.3] we know that for any $s\in S$ there exists an \'etale neighborhood $U$ of $s$ such that $G_{U}$ is split. In particular reductive groups over algebraically closed fields are always split. 
\begin{rqe}
    To give an isomorphism of $S$-group schemes $\Spec(\O_S[X])\cong T$ is equivalent to giving an isomorphism of $S$-group schemes $\underline{X}_S\cong\SheafHom_S(T,\GG_m)$. When discussing splittings, we will choose one of these equivalent formulations, whichever is more convenient in the context.
\end{rqe}

\subsubsection{Type map} Following [XXII, 2.6-2.8], we define a ``type" map $S\to\RD$ that sends a point $s$ to the \emph{type} of $G_s$, i.e. the isomorphism class of the root datum associated to any splitting of the geometric fiber $G_{\Bar{s}}$, where $\Bar{s}$ is any geometric point of $S$ over $s$. This map is locally constant in the Zariski topology. Let $\RD(G)$ be its image. Then we have a disjoint union decomposition 
    \[S=\bigsqcup_{t\in\RD(G)} S_t\]
    where each $S_t$ is the nonempty open and closed subset of $S$ on which the type map has value $t\in\RD(G)$. 
\subsubsection{``Split" model} We let $G_0$ be the reductive group scheme over $S$ such that for each $t\in\RD(G)$, the restriction $G_0|_{S_t}$ is a split reductive group scheme over $S_t$ with root datum $t$. We may and will fix a maximal torus $T_0\subset G_0$ such that $T_0|_{S_t}$ is part of a splitting of $G_0|_{S_t}$ for each $t\in\RD(G)$. If $S$ is connected, then $G_0$ is split but in general $G_0$ may not be split since the types on different connected components of $S$ may be different. 
\subsubsection{Killing pairs}\label{sec:Killing-pair} A \emph{Killing pair} for $G$ is a pair $(T,B)$ where $T\subset G$ is a maximal torus and $B\subset G$ is a Borel subgroup such that $T\subset B$, see [XXII, 5.3.13]. The set of all Killing pairs for $G$ has a natural structure of a setoid\footnote{By a \emph{setoid}, we mean a (small) category in which all morphisms are isomorphisms and any Hom set has cardinality at most one; or equivalently, a groupoid in which all objects have no nontrivial automorphisms.}, denoted by $\Kil(G)$, in which morphisms from an object $(T,B)$ are given by conjugation by elements in $G(S)/T(S)$. A Killing pair $(T,B)$ for $G$ is \emph{split} if the torus $T$ is split. Let $\Kil^{\mathrm{sp}}(G)\subset\Kil(G)$ be the sub-setoid of split Killing pairs. Since splittings exist \'etale locally, so do split Killing pairs. In other words, there exists a surjective \'etale morphism $S'\to S$ such that $\Kil^{\mathrm{sp}}(G_{S'})\ne\varnothing$.

\subsection{Pinnings}\label{sec:pinning}
Recall from [XXIII, D\'efinition 1.1] that a \emph{pinning} of $G$ is a quintuple 
\[\E=(T,X,R,\Delta,\{e_\alpha\}_{\alpha\in\Delta})\] 
where
\begin{itemize}
    \item $(T,X,R)$ is a splitting of $G$;
    \item $\Delta\subset R$ is a system of simple roots;
    \item For each $\alpha\in\Delta$, $e_\alpha\colon\O_S\xrightarrow{\sim}\mathrm{Lie}(G)_\alpha$ is a trivialization of the root line bundle of $\alpha$.
\end{itemize}
\begin{rqe}
    The set of simple roots $\Delta$ determines a Borel subgroup $B\subset G$ such that $(T,B)$ forms a split Killing pair of $G$. But conversely a split Killing pair may not be extended to a pinning since the root line bundles $\mathrm{Lie}(G)_\alpha$ may not be trivial. However, one can pass to a Zariski open cover to trivialize the root line bundles and obtain a pinning. Thus any reductive group scheme admits a pinning \'etale locally on the base. Also, when $S$ is disconnected not all Borel subgroups containing $T$ arise from a set of simple roots $\Delta\subset R$, see \cite[Remark 6.1.2]{ConradReductiveGroupScheme}. So this notion of pinning is more restrictive than what one often encounters in the literature.
\end{rqe}
Any pinning $\E$ gives rise to a reduced based root datum $\Psi=\Psi_\E=(X,X^\vee,\Delta,\Delta^\vee)$, and we say that $\E$ is a \emph{$\Psi$-pinning}. We have a natural action of the automorphism group $\Aut_G(S)$ on the set of $\Psi$-pinnings: an element $\gamma\in\Aut_G(S)$ sends $\E$ to the $\Psi$-pinning
\[\gamma(\E)=(\gamma(T),X,R,\Delta,\{d\gamma(e_{\alpha})\}_{\alpha\in\Delta})\]
in which the underlying splitting has isomorphism $\Spec(\O_S[X])\cong\gamma(T)$ defined as the composition of the isomorphism $\Spec(\O_S[X])\cong T$ from the splitting of $\E$ and the isomorphism $\gamma$. \par 
On the other hand, let $\underline{\Aut(\Psi)}_S$ be the constant group scheme on $S$ associated to the automorphism group of $\Psi$. Then we have a free left action of the group of global sections $\Gamma(S,\underline{\Aut(\Psi)}_S)$, which consists of locally constant $\Aut(\Psi)$-valued functions on $S$, on the set of $\Psi$-pinnings of $G$ defined as follows. An element $v\in\Gamma(S,\underline{\Aut(\Psi)}_S)$ takes the $\Psi$-pinning $\E$ as above to the $\Psi$-pinning 
\[v(\E)=(T,X,R,\Delta,\{e_\alpha'\}_{\alpha\in\Delta})\]
whose underlying splitting has isomorphism $j'=j\circ v^{-1}\colon \underline{X}_S\cong\SheafHom_S(T,\GG_m)$ where $j$ is the isomorphism in the splitting for $\E$, and $e_\alpha'=e_{v^{-1}(\alpha)}$ on every open set where $v$ is constant.
\begin{lem}\label{lem:pinning}
    Let $G$ be a reductive group scheme over $S$ and let $\Psi$ be a reduced based root datum. Suppose there exists a $\Psi$-pinning for $G$. Then the group $G_{\ad}(S)\times\Gamma(S,\underline{\Aut(\Psi)}_S)$ acts simply transitively on the set of $\Psi$-pinnings of $G$ in the way described above. More precisely, let $\E,\E'$ be two $\Psi$-pinnings of $G$. Let $T$ (resp. $T'$) be the maximal torus from $\E$ (resp. $\E'$). Then there exists a unique pair $(u,v)\in G_\ad(S)\times\Gamma(S,\underline{\Aut(\Psi)}_S)$ such that 
    \begin{enumerate}
        \item We have $\ad(u)(T)=T'$ and the following diagram is commutative
        \[\xymatrix{
            \underline{X}_S\ar[r]^{j_{\E'}\qquad}_{\sim\qquad}\ar[d]_{v^{-1}} & \SheafHom_S(T',\GG_m)\ar[d]^{\ad(u)^*}\\
            \underline{X}_S\ar[r]^{j_{\E}\qquad}_{\sim\qquad} & \SheafHom_S(T,\GG_m)
        }\]
        where $j_{\E'}$ (resp. $j_\E$) is the isomorphism in the splitting underlying $\E'$ (resp. $\E$).
        \item On open subsets of $S$ where $v$ is constant, the adjoint action $\ad(u)$ on the Lie algebra sends the set $\{e_\alpha\}_{\alpha\in\Delta}$ for $\E$ to the set $\{e_\alpha'\}_{\alpha\in\Delta}$ for $\E'$. 
    \end{enumerate}
\end{lem}
\begin{proof}
    This is essentially a reformulation of a corrected version of [XXIV, Lemme 1.5]. This result as stated in \emph{loc. cit.} is not literally correct when the base scheme $S$ is disconnected. The issue is that in \emph{loc. cit.}, when translated to our formulation, the group $\Aut(\Psi)$ is used instead of $\Gamma(S,\underline{\Aut(\Psi)}_S)$. More precisely, an isomorphism of pinnings $\E\xrightarrow{\sim}\E'$ in \emph{loc. cit.} is required to be accompanied by an isomorphism of based root data $\Psi'\to\Psi$ that is constant on all connected components of $S$. But this requirement is too restrictive for the statement to be true. For example, consider the case where $S=\Spec(k)\sqcup\Spec(k)$ consists of two points ($k$ is an algebraically closed field) and $G=\GG_m$ on $S$. The based root datum is simply $\Psi=(\ZZ,\ZZ,\emptyset,\emptyset)$ and a $\Psi$-pinning for $G$ is the same as a splitting for $G$, i.e. an isomorphism of $S$-group schemes $\GG_m\cong\GG_m$. Consider the pinning $\E_1$ corresponding to the identity, and the pinning $\E_2$ corresponding to the isomorphism which is identity on one component and inverse map on the other component. The only inner automorphism of $G$ is the identity, and $\E_2$ corresponds to a locally constant family of automorphisms of $\Psi$. The result is corrected by allowing $v$ to be a locally constant $\Aut(\Psi)$-valued function on $S$ as in our formulation.
\end{proof}
Given a $\Psi$-pinning $\mathcal{E}$ of $G$ over $S$, we let $\Aut_G^\E$ be the subgroup scheme of $\Aut_G$ preserving the $\underline{\Aut(\Psi)}$-orbit of $\E$. More precisely, for any test scheme $S'\to S$, $\Aut_G^\E(S')$ is the subgroup of $\Aut_G(S')$ that stabilizes the $\Gamma(S',\underline{\Aut(\Psi)}_S)$-orbit of $\E_{S'}$ (the pinning of $G_{S'}$ induced by $\E$). Then we get isomorphisms of $S$-group schemes: 
\begin{equation}\label{eq:Out-isom}
    \underline{\Aut(\Psi)}_S\cong\Aut^{\mathcal{E}}_G\cong\Out_G
\end{equation}
where $\Out_G\coloneqq\Aut_G/G_{\ad}$ is the outer automorphism group scheme of $G$. The second isomorphism further induces an isomorphism of $S$-groups (see \cite[Theorem 7.1.9]{ConradReductiveGroupScheme}): 
\begin{equation}\label{eq:Aut-splitting}
    \Aut_G\cong G_\ad\rtimes\Out_G.
\end{equation}
In particular $\Out_G$ is a constant group scheme when $G$ is split and in this case it is finite if and only if the maximal central torus of $G$ has rank at most $1$. 

\section{Split reductive monoids over connected base schemes}\label{sec:split-monoid}
\subsection{Review of reductive monoids}
\begin{defi}\label{def:red-monoid}
    Let $S$ be a scheme. A \emph{reductive monoid} over $S$ is a monoid scheme $\M$ over $S$ such that the unit group $\M^\times$ is a reductive group scheme over $S$ and the structure morphism $\M\to S$ is affine, flat, of finite presentation, with integral and normal geometric fibers.\par 
    Let $G$ be a reductive group scheme over $S$. A \emph{reductive $G$-monoid} over $S$ is a reductive monoid $\M$ over $S$ whose unit group scheme $\M^\times$ is isomorphic to $G$. A \emph{rigidified reductive $G$-monoid over $S$} is a pair $(\M,\iota)$ where $\M$ is a reductive monoid scheme over $S$ and $\iota\colon G\xrightarrow{\sim}\M^\times$ is an isomorphism of group schemes over $S$.
\end{defi}

\begin{lem}\label{lem:inner-aut-extend}
    Let $\M$ be a reductive monoid with unit group $G$. The adjoint action of $G_\ad=G/Z_G$ on $G$ extends uniquely to $\M$. 
\end{lem}
\begin{proof}
    According to \cite[Lemma 2.1]{CJ}, the unit group $G$ is universally schematically dense in $\M$. Hence for any scheme $S'\to S$, any automorphism $\M_{S'}\to\M_{S'}$ is uniquely determined by its restriction to the unit group $G_{S'}$. In particular, under the adjoint action of $G$ on $\M$, any element in the center $Z_G(S')$ acts trivially and so the action factors through $G_\ad$. 
\end{proof}
\begin{defi}\label{def:monoid-category}
    Let $G$ be a reductive group scheme over a scheme $S$ and let $S'$ be an $S$-scheme.
    \begin{enumerate}
        \item[(i)] Define $\Mon^\bullet_G(S')$ to be the category whose objects are the rigidified reductive $G_{S'}$-monoids over $S'$, and a morphism between two objects $(\M,\iota)$ and $(\M',\iota')$ is a morphism of monoid schemes $f\colon\M\to\M'$ over $S'$ that commutes with the rigidifications $\iota,\iota'$ on unit groups. From \cite[Lemma 2.1(ii)]{CJ} we know that any such morphism is uniquely determined by its restriction to the unit groups. Therefore $\Mon^\bullet_G(S')$ is a thin category (i.e. a preordered set). 
        \item[(ii)] Define $\Mon^\circ_G(S')$ to be the poset of isomorphism classes of objects in $\Mon^\bullet_G(S')$. Then the natural map $\Mon^\bullet_G(S')\to\Mon^\circ_G(S')$ is an equivalence of thin categories.
        \item[(iii)] Define $\Mon_G(S')$ to be the set of isomorphism classes of reductive monoids over $S'$ whose unit groups are isomorphic to $G_{S'}$.
    \end{enumerate}
    The assignments $S'\mapsto\Mon^\bullet_G(S')$ and $S'\mapsto\Mon^\circ_G(S')$ define presheaves of pre-ordered sets on the category of schemes over $S$. 
    When $S'=\Spec(A)$ is an affine scheme, we also write 
    \[\Mon_G^\bullet(A)\coloneqq\Mon_G^\bullet(\Spec(A)),\qquad\Mon_G(A)\coloneqq\Mon_G(\Spec(A)).\]
\end{defi}
\begin{rqe}
    Let $\Aut_G$ be the automorphism scheme of $G$ over $S$. Then for any scheme $S'\to S$, there is a natural (left) action of $\Aut_G(S')$ on (the underlying set of) $\Mon^\bullet_G(S')$ by changing rigidifications: $\delta\in\Aut_G(S')$ acts by $(\M,\iota)\mapsto(\M,\iota\circ\delta^{-1})$. It is clear that $\Mon_G(S')$ is in bijection with the set of $\Aut_G(S')$-orbits on the set of isomorphism classes of $\Mon_G^\bullet(S')$. This action induces an action of $\Aut_G(S')$ on the set $\Mon^\circ_G(S')$, under which the subgroup $G_\ad(S')$ acts trivially. So $\Mon_G(S')$ is also in bijection with the set of $\Aut_G(S')$-orbits on $\Mon_G^\circ(S')$. 
\end{rqe}

\begin{defi}\label{def:Weyl-cone}
    Let $\Phi=(X,X^\vee,R,R^\vee)$ be a root datum and let $W$ be its Weyl group. A \emph{Weyl cone} in $X_\QQ\coloneqq X\otimes_\ZZ\QQ$ is a rational convex polyhedral cone $C\subset X_\QQ$ of full dimension that is stable under the action of $W$. Let $\Cone_{\mathrm{fd}}(X_\QQ)$ denote the poset of rational convex polyhedral cones of full dimension in $X_\QQ$, with the partial order given by \emph{reverse} inclusion of cones. Let $\Cone^W(X_\QQ)\subset\Cone_{\mathrm{fd}}(X_\QQ)$ denote the sub-poset of Weyl cones.\par 
    Dually, we let $\Cone_{\mathrm{pt}}(X_\QQ^\vee)$ denote the poset of pointed\footnote{Recall that a convex cone $C$ in a $\QQ$-vector space is \emph{pointed}, or \emph{strictly convex}, if $C\cap(-C)=\{0\}$.} rational convex polyhedral cones in $X^\vee_\QQ$, on which the partial order is given by inclusion of cones; and let $\Cone^W(X^\vee_\QQ)\subset\Cone_{\mathrm{pt}}(X_\QQ^\vee)$ denote the sub-poset of \emph{dual Weyl cones}, i.e. the pointed rational convex polyhedral cones that are $W$-stable. 
\end{defi}
\begin{rqe}
    Cone duality induces isomorphisms of posets:
    \[\Cone_{\mathrm{fd}}(X_\QQ)\cong\Cone_{\mathrm{pt}}(X_\QQ^\vee),\qquad \Cone^W(X_\QQ)\cong\Cone^W(X^\vee_\QQ).\]
    Our main results will be stated in terms of cones in $X_\QQ$. In \cite{Ngo-Hankel}, similar results (over a field) are formulated using dual cones in $X^\vee_\QQ$.
\end{rqe}
Using this duality we get the following explicit description of Weyl cones.
\begin{lem}\label{lem:Weyl-cone}
    Let $K(X^\vee,W)$ be the set of all finite subsets $\Omega\subset X^\vee$ satisfying the following conditions:
    \begin{enumerate}
        \item Any element $\xi\in\Omega$ is nonzero and primitive in $X^\vee$. In other words, $\xi$ cannot be written as $n\eta$ where $\eta\in X^\vee$ and $n>1$ is an integer.
        \item The rational convex polyhedral cone $\sigma_\Omega\subset X^\vee_\QQ$ generated by $\Omega$ is pointed.
        \item $\Omega$ is irredundant as a generating set of $\sigma_\Omega$: 
        \[\xi\notin\sum_{\eta\in\Omega\setminus\{\xi\}}\QQ_{\ge0}\eta,\qquad\forall\xi\in\Omega.\]
        \item $\Omega$ is $W$-stable: for all $w\in W$ we have $w(\Omega)=\Omega$. 
    \end{enumerate}
    Then the map $K(X^\vee,W)\to\Cone^W(X_\QQ)$ sending $\Omega\in K(X^\vee,W)$ to the dual cone $\sigma_\Omega^\vee\subset X_\QQ$ defined by
    \[\sigma_\Omega^\vee=\{x\in X_\QQ\mid \xi(x)\ge0,\;\forall\xi\in\Omega\}\]
    is a bijection. In particular, the empty set $\emptyset$ is an element of $K(X^\vee,W)$ and maps to the cone $X_\QQ\in\Cone^W(X_\QQ)$ (which is the minimal element under the reverse inclusion partial order). 
\end{lem}
\begin{proof}
    The map is well-defined by conditions (2) and (4). Indeed, the condition that $\sigma_\Omega$ is pointed is equivalent to the condition that the dual cone $\sigma_\Omega^\vee$ spans $X_\QQ$. The map is injective by condition (1) and (3). It remains to prove surjectivity. Let $C\in\Cone^W(X_\QQ)$ be a Weyl cone. Then its dual cone $C^\vee\subset X^\vee_\QQ$ is a pointed rational convex polyhedral cone. Let $\Omega_C\subset C^\vee\cap X^\vee$ be the set of primitive elements in the extremal rays of $C^\vee$. Then it is clear that $\Omega_C$ satisfies conditions (1), (2), (3). Since $W$ preserves $C^\vee$ and the lattice $X^\vee$, it permutes the extremal rays of $C^\vee$ and their primitive lattice generators. Thus $\Omega_C$ is $W$-stable and we conclude that $\Omega_C\in K(X^\vee,W)$.
\end{proof}
\begin{defi}\label{def:canonical-generator-Weyl-cone}
    For each Weyl cone $C\in\Cone^W(X_\QQ)$, the set $\Omega\in K(X^\vee,W)$ with $\sigma_\Omega^\vee=C$ from Lemma \ref{lem:Weyl-cone} is called the \emph{set of canonical facet normals} of $C$. The \emph{sum of canonical facet normals of $C$} is the element
    \[\omega(C)\coloneqq\sum_{y\in\Omega}y\in (X^\vee)^W.\]
    It is $W$-invariant since $\Omega$ is $W$-stable. Therefore we obtain a map:
    \[\omega\colon\Cone^W(X_\QQ)\to (X^\vee)^W,\quad C\mapsto\omega(C)\]
    with the following properties (both are easily verified):
    \begin{itemize}
        \item $\omega(C)=0$ if and only if $C=X_\QQ$. 
        \item $\omega$ is equivariant under the natural action of $\Aut(\Psi)$ on both sides.
    \end{itemize}
\end{defi}

\subsection{Classification of split reductive monoids over connected base}\label{sec:split-classification}
Let $S$ be a nonempty and connected base scheme. Let $G$ be a reductive group over $S$ with a pinning $\E$ that gives rise to a splitting $(T,X,R)$ and let $W=N_G(T)/T$ be the Weyl group. Let $j_\E\colon \underline{X}_S\cong\SheafHom_S(T,\GG_m)$ be the isomorphism in the splitting data. Since $S$ is nonempty and connected, it induces an isomorphism that we still denote by $j_\E\colon X\cong X^*(T)$. For any $(\M,\iota)\in\Mon_G^\bullet(S)$, we let $\M_T$ be the schematic closure of $\iota(T)$ in $\M$. According to \cite[Theorem 2.19]{CJ}, the pair $(\M_T,\iota|_T)\in\Mon^\bullet_T(S)$ is a rigidified reductive $T$-monoid on $S$, and if we let $C_\E(\M,\iota)\subset X_\QQ$ be the cone generated by the characters $\chi\in X$ such that $j_\E(\chi)\in X^*(T)$ extends to a regular function on $\M_T$ along $\iota|_T$, then we have $C_\E(\M,\iota)\in\Cone^W(X_\QQ)$. 

\begin{theo}\label{thm:split-classification}
    Let the notations and assumptions be as above ($G$ has pinning $\E$ and the base scheme $S$ is nonempty and connected). 
    \begin{enumerate}
        \item The assignment $(\M,\iota)\mapsto C_\E(\M,\iota)$ defines an equivalence of thin categories
        \[C_\E\colon\;\Mon_G^\bullet(S)\simeq\Cone^W(X_\QQ).\]
        \item For any $\gamma\in\Aut_G(S)$ with image $o(\gamma)\in\Out_G(S)$, the following diagram is commutative
            \[\xymatrix{
            \Mon_G^\bullet(S)\ar[d]^{\gamma}\ar[r]^{C_\E} & \Cone^W(X_\QQ)\ar[d]^{o(\gamma)}\\
            \Mon_G^\bullet(S)\ar[r]^{C_\E} & \Cone^W(X_\QQ)
            }\]
            where the left arrow is defined by pre-composing a rigidification with $\gamma^{-1}$, and the right arrow is induced, via the isomorphism $j_\E\colon X\cong X^*(T)$, by the action of $o(\gamma)$ on $X^*(T)$ via inverse pullback (using the canonical isomorphism $\Out_G\cong\Aut_G^\E$). 
    \end{enumerate}
\end{theo}
\begin{proof}
    (1) This follows from the combination of Theorem 3.7 and Corollary 2.14 in \cite{CJ}.\par 
    (2) Factorize the automorphism into $\gamma=\delta\circ\mathrm{\ad}(u)$ where $u\in G_{\ad}(S)$ and $\delta\in\Aut^\E_G(S)$. By Lemma \ref{lem:inner-aut-extend}, the inner automorphism $\ad(u)$ on $G$ extends (through the rigidification $\iota$) to an automorphism of $\M$. Then we have $\M_{\ad(u)(T)}=\ad(u)(\M_T)$ and so a character $j_\E(\chi)\in X^*(T)$, where $\chi\in X$, extends to a regular function on $\M_{\ad(u)^{-1}(T)}$ along $\iota\circ\ad(u)^{-1}|_T$ if and only if $j_\E(\chi)$ extends to a regular function on $\M_T$ along $\iota|_T$. Put $\iota'\coloneqq\iota\circ\ad(u)^{-1}$. Then we have shown that
    \[C_\E(\M,\iota')=C_\E(\M,\iota).\]
    On the other hand we have $\iota\circ\gamma^{-1}=\iota'\circ\delta^{-1}$ and $\delta$ preserves $T$. Hence we get
    \[C_\E(\M,\iota\circ\gamma^{-1})=(\delta^{-1})^*C_\E(\M,\iota')=(\delta^{-1})^*C_\E(\M,\iota).\]
    
    Finally we note that under the isomorphism $\Out_G\cong\Aut_G^\E$, the element $o(\gamma)$ is mapped to $\delta$ and so the right vertical arrow $o(\gamma)$ is induced by the composition of the following isomorphisms:
    \[X\xrightarrow{j_\E}X^*(T)\xrightarrow{(\delta^{-1})^*} X^*(T)\xrightarrow{j_\E^{-1}}X\]
    and this finishes the proof.
\end{proof}

\begin{cor}\label{cor:monoid-automorphism-field}
    Let $G$ be a split reductive group scheme over a field $k$ and let $\E$ be a pinning of $G$. For any rigidified reductive $G$-monoid $(\M,\iota)\in\Mon^\bullet_G(k)$ with Weyl cone $C=C_\E(\M,\iota)\subset X_\QQ$, the morphism of $k$-group schemes $\Aut_\M\to\Aut_G$ determined by the rigidification $\iota$ induces a canonical isomorphism
    \[\Aut_\M(k)\cong G_\ad(k)\rtimes\Out_{G}(k)_C\]
    where $\Out_{G}(k)_C$ denotes the subgroup of $\Out_G(k)$ that stabilizes the set $C$.
\end{cor}
\begin{proof}
    Recall that the pinning $\E$ determines the isomorphism \eqref{eq:Aut-splitting}. By Lemma \ref{lem:inner-aut-extend} we have $G_\ad\subset\Aut_\M$. Let $\theta\in\Out_G(k)$ and suppose that $\theta\in\Aut_\M(k)$. Since $\theta(T)=T$, we have $\theta(\M_T)=\M_T$ and therefore $\theta\in\Out_{G}(k)_C$.
    Conversely if $\theta(C)=C$, then the two objects $(\M,\iota)$ and $(\M,\iota\circ\theta)$ in $\Mon_G^\bullet(k)$ are isomorphic by Theorem \ref{thm:split-classification}. This means that $\theta$ extends to an automorphism of $\M$ (as monoid scheme over $k$). 
\end{proof}
This result will be generalized in Corollary \ref{cor:Aut(M)}.
\section{Scheme of based root data and Weyl cones}\label{sec:RD-Cone-scheme}
To obtain a classification theorem for general reductive monoids from Theorem \ref{thm:split-classification}, there are several issues that one needs to resolve: 
\begin{enumerate}
    \item Extend the split classification to the case of disconnected base schemes;
    \item Examine the effect of changing pinnings;
    \item Descend from the split case to get the general result.
\end{enumerate}
To achieve these, we are forced to formulate everything in terms of \emph{canonical} objects associated to a reductive group scheme, that are free of any specific choices. In this section we define these canonical objects and resolve problems (1) and (2) above. 

\subsection{Absolute based root data}\label{sec:absolute-RD}
As a warm up, we first recall the construction in the case of fields.\par 
Let $k$ be a field, $k_s$ be a separable closure of $k$ and $\Gamma\coloneqq\mathrm{Gal}(k_s/k)$ be the absolute Galois group. Let $G$ be a reductive group over $k$. We can define a canonical based root datum from $G$, following \cite[Remark 7.1.2]{ConradReductiveGroupScheme}.
\begin{defi}\label{def:absolute-RD}
    The \emph{absolute based root datum} of $G$, denoted $\Psi=\Psi_G=(X,X^\vee,\Delta,\Delta^\vee)$, is defined as the colimit over all (split) Killing pairs $(T,B)$ of $G_{k_s}$ with $T\subset B$, of the based root datum $(X^*(T),X_*(T),\Delta_B,\Delta_B^\vee)$ where $\Delta_B$ is the set of simple roots determined by $B$ and $\Delta_B^\vee$ is the corresponding set of simple coroots. The transition maps in this system are bijections induced by conjugation by elements in $G_{k_s}$.
\end{defi}
Note that $\Psi$ only depends on the choice of the separable closure $k_s$. Let $\Aut(\Psi)$ be the automorphism group of the based root data $\Psi$. Then there is a canonical group homomorphism 
\begin{equation}\label{eq:rho-G}
    \rho_G\colon\;\Gamma\to\Aut(\Psi).
\end{equation}
Let $W=W_\Psi$ be the Weyl group of $\Psi$. Then $\rho_G$ induces an action of $\Gamma$ on $W$ and any (split) Killing pair $(T_{k_s}, B_{k_s})$ of $G_{k_s}$ determines an isomorphism $W\cong N_{G_{k_s}}(T_{k_s})/T_{k_s}$.

\subsection{Schemes of based root data, abstract Cartan and Weyl groups}
Let $S$ be a nonempty scheme and let $G$ be a reductive group scheme over $S$. 
\begin{defi}\label{def:sheaf-weight-lattice}
    Let $\underline{X}^G$ (resp. $\underline{X}^{G,\vee}$) be the sheaf of abelian groups on the big \'etale site of $S$ associated to the presheaf that sends an $S$-scheme $S'$ to the colimits of abelian groups
    \[\varinjlim_{(T,B)\in\Kil^{\mathrm{sp}}(G_{S'})}X^*(T),\qquad\left(\text{resp. }\quad \varinjlim_{(T,B)\in\Kil^{\mathrm{sp}}(G_{S'})}X_*(T)\right)\]
    in which the transition maps are given by
    \[\Ad(g^{-1})^*\colon X^*(T)\xrightarrow{\sim}X^*(T'),\qquad\left(\text{resp. }\quad \Ad(g)_*\colon X_*(T)\xrightarrow{\sim}X_*(T')\right)\]
    for any coset in $G(S')/T(S')$ represented by an element $g\in G(S')$ and any split Killing pair $(T,B),(T',B')\in\Kil^{\mathrm{sp}}(G_{S'})$ with $\Ad(g)(T,B)=(T',B')$. The map is independent of the choice of representative since $T(S')$ acts trivially.\par 
    Applying the functor $\text{-}\otimes_\ZZ\QQ$ to each term in the above colimit, we get \'etale sheaves of $\QQ$-vector spaces $\underline{X}^G_\QQ$ and $\underline{X}^{G,\vee}_\QQ$ on $S$.
\end{defi}
\begin{rqe}
    Over an algebraically closed field, one can also use a limit in the above definition since the setoid of Killing pairs is connected and all transition maps in the system are isomorphisms (leading to a canonical identification between limit and colimit). The reason to use colimit instead of limit in general is that sheafification preserves colimits (since it is a left adjoint).
\end{rqe}

\begin{lem}\label{lem:Killing-pair-compatibility}
    Let $S'$ be a scheme over $S$ such that $\Kil^{\mathrm{sp}}(G_{S'})$ is nonempty. Then any split Killing pair $(T,B)\in\Kil^{\mathrm{sp}}(G_{S'})$ determines canonical isomorphisms of sheaves of abelian groups on $S'$:
    \[\iota_{(T,B)}\colon\SheafHom_{S'}(T,\GG_m)\cong\underline{X}^G|_{S'},\quad\iota_{(T,B)}^\vee\colon\SheafHom_{S'}(\GG_m,T)\cong\underline{X}^{G,\vee}|_{S'}.\]
    Moreover, for any $\gamma\in G_\ad(S')$, let $(T',B')\coloneqq\ad(\gamma)(T,B)$, then the following diagram commutes:
    \[\xymatrix{
        \SheafHom_{S'}(T',\GG_m)\ar[r]^{\qquad\iota_{(T',B')}}\ar[d]_{\ad(\gamma)^*} & \underline{X}^G|_{S'}\ar[d]^{=}\\
        \SheafHom_{S'}(T,\GG_m)\ar[r]^{\qquad\iota_{(T,B)}} & \underline{X}^G|_{S'}
    }\]
\end{lem}
\begin{proof}
    The existence of these maps follows by definition. To check the isomorphism and compatibility we can work \'etale locally on every test scheme $U\to S'$. 
    Since any two Killing pairs become conjugate \'etale locally, the colimit identifies locally with the sheaf of characters and cocharacters of the chosen $T$. This shows that the maps are isomorphisms and compatible with the actions of $G_{\ad}(S')$.
\end{proof}

In particular any pinning $\E$ of $G_{S'}$ (if exists), with associated splitting $(T_\E,X_\E,R_\E)$, determines isomorphisms
\[\iota_\E\colon\underline{X_\E}_{S'}\xrightarrow{\sim}\underline{X}^G|_{S'},\quad\iota_\E^\vee\colon\underline{X_\E^\vee}_{S'}\xrightarrow{\sim}\underline{X}^{G,\vee}|_{S'}.\]
For the following result, let us recall the Dynkin scheme $\underline{\mathrm{Dyn}}_G$ from \cite[XXIV.3.3]{SGA3-3new}. Roughly speaking, it is obtained by \'etale descent from the constant scheme associated to the set of simple roots (or coroots) on an \'etale cover on which pinnings exist.
\begin{prop}\label{prop:dynkin-scheme-embedding}
    The sheaves $\underline{X}^G$ and $\underline{X}^{G,\vee}$ are represented by quasi-isotrivial ``constant tordu"\footnote{see [\emph{loc. cit.}, X, D\'efinition 5.1]} group schemes over $S$. Moreover, there exist closed embeddings of schemes $\underline{\Dyn}_G\into\underline{X}^G$ and $\underline{\Dyn}_G\into\underline{X}^{G,\vee}$ such that for any $S$-scheme $S'$ and any pinning $\E$ of $G_{S'}$ (if one exists), with associated splitting $(T_\E,X_\E,R_\E)$ and simple roots (resp. coroots) $\Delta_\E$ (resp. $\Delta_\E^\vee$), the compositions
    \[\underline{\Dyn}_{G,S'}\into\underline{X}^G|_{S'}\xrightarrow{\iota_\E^{-1}}\underline{X_\E}_{S'},\quad (\text{resp.  }\underline{\Dyn}_{G,S'}\into\underline{X}^{G,\vee}|_{S'}\xrightarrow{(\iota_\E^\vee)^{-1}}\underline{X^\vee_\E}_{S'})\]
    have image $\underline{\Delta_\E}_{S'}$ (resp. $\underline{\Delta_\E^\vee}_{S'}$). 
\end{prop}
\begin{proof}
    Working separately on the open and closed type strata $S_t$, we may assume that the type of $G$ is constant on $S$. Let $S'\to S$ be a surjective \'etale morphism such that $G_{S'}$ has a pinning (see \S\ref{sec:pinning}). The pullback of $\underline{X}^G$ and $\underline{X}^{G,\vee}$ to $S'$ are represented by constant group schemes over $S'$ as we saw above. Then according to \cite[X, Corollaire 5.5]{SGA3VIIIaXVIII}, these sheaves are represented by quasi-isotrivial ``constant tordu" group schemes over $S$. The second statement on embeddings of Dynkin schemes follows from the construction of $\underline{\Dyn}_{G,S'}$ as a colimit over all pinnings $\E$ of $G_{S'}$ of the simple roots $\Delta_\E$. 
\end{proof}

\begin{defi}[Abstract based root data]\label{def:Abstract-RD-scheme}
    Let $\underline{\Delta}^G$ (resp. $\underline{\Delta}^{G,\vee}$) be the image of $\underline{\Dyn}_G$ in $\underline{X}^G$ (resp. $\underline{X}^{G,\vee}$) under the closed embeddings from Proposition \ref{prop:dynkin-scheme-embedding}. Then there is a canonical isomorphism $\underline{\Delta}^G\cong\underline{\Delta}^{G,\vee}$. The quadruple $(\underline{X}^G,\underline{X}^{G,\vee},\underline{\Delta}^G,\underline{\Delta}^{G,\vee})$ is called the scheme of \emph{abstract based root data for $G$}. Its fiber at each geometric point $\Bar{s}\in S$ is the absolute root datum of $G_{\Bar{s}}$ from Definition \ref{def:absolute-RD}.
\end{defi}

\begin{defi}[The abstract Cartan]\label{def:abstract-Cartan}
    Let $T_G$ be the sheaf of abelian groups on the big \'etale site of $S$ associated to the presheaf that sends an $S$-scheme $S'$ to the group 
    \[\varinjlim_{(T,B)\in\Kil^{\mathrm{sp}}(G_{S'})}T(S')\]
    in which the transition maps are given by conjugation by elements in $G(S')/T(S')$. It is represented by the quasi-isotrivial torus over $S$, called \emph{the abstract Cartan for $G$}, whose sheaf of weight lattice is $\SheafHom_S(T_G,\GG_m)=\underline{X}^G$.  By construction there is a natural embedding $Z_G\into T_G$ and the quotient 
    \[T_G/Z_G\cong T_{G_{\ad}}\]
    is the abstract Cartan for the adjoint group $G_\ad=G/Z_G$.
\end{defi}

\begin{defi}[The abstract Weyl group]
    Let $W_G$ be the sheaf (of sets, a priori) on the big \'etale site of $S$ associated to the presheaf that sends an $S$-scheme $S'$ to the set
    \[\varinjlim_{(T,B)\in\Kil^{\mathrm{sp}}(G_{S'})}N_G(T)(S')/T(S')\]
    in which the transition maps are given by conjugation by elements in $G(S')/T(S')$. Since $G$ admits a split maximal torus \'etale locally on $S$, the sheaf $W_G$ is \'etale locally represented by a constant finite group scheme and hence by \cite[X.5.5]{SGA3VIIIaXVIII} it is represented by a finite \'etale group scheme over $S$, called \emph{the abstract Weyl group of $G$}. By construction $W_G$ acts on the abstract Cartan $T_G$ and there is a natural closed embedding of schemes $\underline{\mathrm{Dyn}}_G\into W_G$ whose image on each geometric fiber is a set of simple reflections.
\end{defi}

We have a natural action of the automorphism group scheme $\Aut_G$ on all the schemes defined above: for any scheme $S'\to S$ such that a pinning $\E$ of $G_{S'}$ exists and any element $\delta\in\Aut_G(S')$, the isomorphism $(\delta^{-1})^*\colon X_\E\to X_{\delta(\E)}$ restricts to a bijection $\Delta_\E\xrightarrow{\sim}\Delta_{\delta(\E)}$; and when passing to the colimit and sheafification we get an automorphism of $\underline{X}^G(S')$ that induces a bijection on $\underline{\Delta}^G(S')$. The actions on $T_G$ and $W_G$ are described in a similar way. In particular, we note that the subgroup of inner automorphisms $G_\ad$ acts trivially on all these schemes and so the action factors through the quotient group scheme $\Out_G$. Moreover, the natural $W_G$-actions on $\underline{X}^G$ and $\underline{X}^{G,\vee}$ (and hence also their $\QQ$-versions) are compatible with the $\Aut_G$-action on all these schemes.\par 

All the schemes we have defined so far are canonically associated to the group scheme $G$ over $S$. Now let us describe them more concretely by making some extra choices. Recall from \S\ref{sec:group-splitting} that we have the partition by type $S=\bigsqcup_{t\in\RD(G)} S_t$. Let $G_0$ be the reductive group scheme over $S$ such that $G_0|_{S_t}$ is the split form of $G|_{S_t}$ for each $t\in\RD(G)$. We have the right $\Aut_{G_0}$ torsor $\SheafIsom_{S}(G_0,G)$ and after quotient by the inner automorphism group $G_{0,\ad}$ we get a right $\Out_{G_0}$-torsor
\[\SheafIsom^{\Out}_{S}(G_0,G)\coloneqq\SheafIsom_{S}(G_0,G)/G_{0,\ad}.\]
We can use this torsor to describe the schemes above. For this we may restrict to each open and closed subscheme $S_t$ and hence assume that the type of $G$ is constant. Then the sheaf of based root data for $G_0$ is constant. Let $\Psi=(X,X^\vee,\Delta,\Delta^\vee)$ be its fiber over a geometric point $\Bar{s}\in S$ and let $W_\Psi$ be the Weyl group of the root system underlying $\Psi$. The isomorphism class of $\Psi$ is independent of the choice of $\Bar{s}$, and if $S$ is connected the fibers at different geometric points are uniquely identified through the global sections. Then we have an isomorphism of constant group schemes $\underline{\Aut(\Psi)}_{S}\cong\Out_{G_0}$ and isomorphisms of sheaves
\begin{align*}
    (\underline{X}^G,\underline{X}^{G,\vee},\underline{\Delta}^G,\underline{\Delta}^{G,\vee})&\cong \SheafIsom_{S}(G_0,G)\times^{\Aut_{G_0}}\Psi\\
    &\cong\SheafIsom^{\Out}_{S}(G_0,G)\times^{\Out_{G_0}}\Psi
\end{align*}
where $\Psi$ is understood as a constant sheaf. Similarly we have isomorphisms of $S$-group schemes
\[T_G\cong\SheafIsom^{\Out}_{S}(G_0,G)\times^{\Out_{G_0}}T_{G_0},\qquad W_G\cong\SheafIsom^{\Out}_{S}(G_0,G)\times^{\Out_{G_0}} W_\Psi\]
where $T_{G_0}$ is the abstract Cartan of $G_0$.\par 
This shows that the abstract Cartan, abstract Weyl group and sheaf of based root data for $G$ coincide with those for its quasi-split inner form.

\subsection{Scheme of Weyl cones}
We would like to construct a scheme over $S$ whose fiber at a geometric point $\Bar{s}\in S$ parametrizes the set of Weyl cones in $\underline{X}^G_{\QQ,\Bar{s}}$. Then it will have the extra structure of a sheaf of posets. First let us present the most straightforward approach that utilizes the $\Out_{G_0}$-torsor $\SheafIsom^{\Out}_S(G_0,G)$ from above. 
\begin{defi}
    If the type of $G$ is constant on $S$, we define the \emph{scheme of Weyl cones for $G$} to be the following scheme over $S$:
    \begin{equation}\label{eq:cone-sheaf-description}
        \SheafCone^W(\underline{X}_\QQ^G)=\SheafIsom^{\Out}_{S}(G_0,G)\times^{\Out_{G_0}}\Cone^W(X_\QQ)
    \end{equation}
    In general we obtain $\SheafCone^W(\underline{X}_\QQ^G)$ by gluing the above construction on each constant type stratum $S_t$ for $t\in\RD(G)$. By construction, it represents an fpqc sheaf of posets on $S$. Moreover, the group scheme $\Aut_G$ acts naturally on $\SheafCone^W(\underline{X}_\QQ^G)$ preserving the poset structure, with the subgroup of inner automorphisms $G_\ad$ acting trivially and so the action factors through $\Out_G$. 
\end{defi}
In particular, if $G$ is an inner form of $G_0$ then the $S$-scheme $\SheafIsom^{\Out}_{S}(G_0,G)$ admits a global section (see \cite[XXIV, Remarque 1.11]{SGA3-3new}) and hence $\SheafCone^W(\underline{X}_\QQ^G)$ is a constant scheme over $S$.\par
We have a natural morphism of $S$-schemes
\begin{equation}\label{eq:omega_G}
    \omega_G\colon\SheafCone^W(\underline{X}_\QQ^G)\to(\underline{X}^{G,\vee})^{W_G} 
\end{equation}
which is obtained, on each constant type stratum $S_t$, from the $\Aut(\Psi)$-equivariant map of taking sum of canonical facet normals $\omega\colon\Cone^W(X_\QQ)\to (X^\vee)^W$ in Definition \ref{def:Weyl-cone}.\par
Next we briefly introduce a more intrinsic definition\footnote{We will not use this alternative definition in the following.} of $\SheafCone^W(\underline{X}_\QQ^G)$ that does not use the $\Out_{G_0}$-torsor construction. The idea is to apply Lemma \ref{lem:Weyl-cone}. As a warm up, for any locally constant sheaf $\F$ on the big \'etale site of $S$, we define the sheaf of finite subsets of $\F$ by
\[\mathrm{Fin}(\F)\coloneqq\bigsqcup_{n\ge0}\mathrm{Conf}_n(\F)/\mathfrak{S}_n\]
where $\mathrm{Conf}_n(\F)\subset\F^n$ is the subsheaf whose geometric fibers are pairwise distinct $n$-tuples on which the symmetric group $\mathfrak{S}_n$ acts freely, and $\mathrm{Conf}_0(\F)=S$ is the final object in the category of \'etale sheaves of sets on $S$ (it parametrizes the ``empty subset" of $\F$). Then $\mathrm{Fin}(\F)$ is a locally constant sheaf on the big \'etale site of $S$ and for any scheme $U\to S$ we have
\[\mathrm{Fin}(\F)(U)=\{\text{Locally constant subsheaf }\L\subset\F|_U\text{ with finite fibers}\}.\]
Based on Lemma \ref{lem:Weyl-cone}, we define
\begin{align*}
    \SheafCone^W(\underline{X}_\QQ^G)(U)=\{\text{Finite locally constant subsheaf }\L\subset\underline{X}^{G,\vee}|_U\colon\\ \L_{\Bar{u}}\in K(\underline{X}^{G,\vee}_{\Bar{u}},W_{\Bar{u}}),\;\forall\Bar{u}\in U\}
\end{align*}
where $W_{\Bar{u}}$ is the fiber of the abstract Weyl group $W_G$ at $\Bar{u}$. This coincides with the previous definition using $\Out_{G_0}$-torsors.

\subsection{Split reductive monoids over disconnected base}\label{sec:main-thm-split-case}
In this subsection we extend the split classification in Theorem \ref{thm:split-classification} to disconnected base and examine the effect of changing pinnings.\par  
Assume that $G$ is a split reductive group over a nonempty scheme $S$. Fix a pinning $\E$ of $G$ with reduced based root data $\Psi$ and splitting $(T,X,R)$. Then $\E$ determines an isomorphism of sheaves 
\[\iota_\E\colon\underline{\Cone^W(X_\QQ)}_S\xrightarrow{\sim}\SheafCone^W(\underline{X}^G_\QQ)\]
where the left hand side is the constant sheaf on $S$ with fiber $\Cone^W(X_\QQ)$. The set of global sections $\Gamma(S,\underline{\Cone^W(X_\QQ)}_S)$ is by definition the set of Zariski locally constant functions on $S$ with value in $\Cone^W(X_\QQ)$. For any $(\M,\iota)\in\Mon_G^\bullet(S)$, from Theorem \ref{thm:split-classification} we obtain a function $\varphi_\M\colon S\to\Cone^W(X_\QQ)$ that is constant on each connected component of $S$. To define the desired functor, we need to show that $\varphi_\M$ is locally constant in the Zariski topology\footnote{Note that a function on $S$ that is constant on each connected component may fail to be locally constant (for the Zariski or \'etale topology), since the connected components of $S$ may not be open in $S$. Let us mention the following fact (though we will not use it): If $S$ is locally noetherian, then all its connected components are open (see \cite[X, Rappel 5.11.0]{SGA3-2new}).}. Choose any geometric point $\Bar{s}\in S$ and an affine open subscheme $U\subset S$ containing $\Bar{s}$. Let $E\subset S$ be the connected component of $U$ containing $\Bar{s}$. Let $C=\varphi_\M(\Bar{s})\in\Cone^W(X_\QQ)$ and let $(\M_C,\iota_C)\in\Mon_G^\bullet(S)$ be the base change to $S$ of the rigidified reductive monoid over $\Spec(\ZZ)$ associated to the cone $C$ by Theorem \ref{thm:split-classification}. Then we have an isomorphism of monoid schemes over $E$:
\[\M|_E\cong\M_C|_E\]
that is compatible with the rigidifications. By \cite[\href{https://stacks.math.columbia.edu/tag/00EB}{Section 00EB}]{Stacks}, the component $E$ is a closed subscheme of $S$ and can be written as the inverse limit of certain open and closed affine neighborhoods $U_j\subset U$. Since $\M|_U$ and $\M_C|_U$ are of finite presentation over $U$, the isomorphism above and its inverse extend to some open and closed affine neighborhood $U_j$. The identities expressing that they are inverses, preserve multiplication and unit, and agree with the rigidifications also descend after increasing $j$. See for example \cite[\href{https://stacks.math.columbia.edu/tag/01ZM}{Lemma 01ZM}]{Stacks}. Thus the value of $\varphi_\M$ is constant on the affine open neighborhood $U_j$ of $\Bar{s}$. We conclude that there are well-defined functors
\[\Mon_G^\bullet(S)\xrightarrow{C_\E}\Gamma(S_{\mathrm{\acute{e}t}}, \underline{\Cone^W(X_\QQ)}_S)\xrightarrow[\sim]{\iota_\E}\Gamma(S_{\mathrm{\acute{e}t}},\SheafCone^W(\underline{X}_\QQ^G))\]
where the second one is the equivalence induced by the sheaf isomorphism $\iota_\E$. The functor $C_\E$ is faithful since both are thin categories. To see its fullness and essential surjectivity, we can restrict to open subsets where the cone-valued function is constant. Then the groups and monoids are obtained by base change from $\Spec(\ZZ)$ and we conclude by Theorem \ref{thm:split-classification}. Thus the functor $C_\E$ is also an equivalence of thin categories that is moreover compatible with the $\Aut_G(S)$-actions as in Theorem \ref{thm:split-classification}(2).\par  
Next we examine the effect of changing pinnings. Let $\E'$ be another $\Psi$-pinning of $G$ giving rise to a splitting $(T',X,R)$. By Lemma \ref{lem:pinning}, there exists $(u,v)\in G_{\ad}(S)\times\Gamma(S,\underline{\Aut(\Psi)}_S)$ such that $(u,v)\cdot\E=\E'$. Let $(T,B)$ (resp. $(T',B')$) be the split Killing pair determined by $\E$ (resp. $\E'$). Then we have $\ad(u)(T)=T'$ and $\ad(u)(B)=B'$. 
We have the following commutative diagram:
\begin{equation}\label{eq:change-pinning-X}
    \xymatrix{
    \underline{X}_S\ar[r]^{j_{\E'}\qquad}_{\sim\qquad}\ar[d]_{v^{-1}} & \SheafHom_S(T',\GG_m)\ar[r]^{\qquad\iota_{(T',B')}}\ar[d]^{\ad(u)^*} & \underline{X}^G\ar[d]^{=} \\
    \underline{X}_S\ar[r]^{j_{\E}\qquad}_{\sim\qquad}  &  \SheafHom_S(T,\GG_m)\ar[r]^{\qquad\iota_{(T,B)}} & \underline{X}^G
    }
\end{equation}

where the isomorphism $j_{\E'}$ (resp. $j_\E$) is part of the splitting data in $\E'$ (resp. $\E$), the left square commutes by Lemma \ref{lem:pinning} and the right square commutes by Lemma \ref{lem:Killing-pair-compatibility}.
\[\xymatrix{
\Mon_G^\bullet(S)\ar[r]^{C_{\E'}\qquad}\ar[d]_{u^{-1}} & \Gamma(S_{\mathrm{\acute{e}t}}, \underline{\Cone^W(X_\QQ)}_S)\ar[r]^{\iota_{\E'}}\ar[d]^{v^{-1}} & \Gamma(S_{\mathrm{\acute{e}t}},\SheafCone^W(\underline{X}^G_\QQ))\ar[d]^{=}\\
\Mon_G^\bullet(S)\ar[r]^{C_\E\qquad} & \Gamma(S_{\mathrm{\acute{e}t}}, \underline{\Cone^W(X_\QQ)}_S)\ar[r]^{\iota_{\E}} & \Gamma(S_{\mathrm{\acute{e}t}},\SheafCone^W(\underline{X}^G_\QQ))
}\]
The commutativity of the right square follows from the commutativity of the outer square in \eqref{eq:change-pinning-X}. Next we examine the left square. Let $(\M,\iota)\in\Mon_G^\bullet(S)$. Then $u^{-1}(\M,\iota)=(\M,\iota\circ\ad(u))$. A section $\chi\in\Gamma(S,\underline{X}_S)$ lies in $C_{\E'}(\M,\iota)$ if and only if $j_{\E'}(\chi)$ extends to a regular function of $\overline{\iota(T')}=\overline{\iota\circ\ad(u)(T)}$ along $\iota|_{T'}$, which is to say that $\ad(u)^*(j_{\E'}(\chi))=j_{\E}(v(\chi))$ extends along $\iota\circ\ad(u)|_T$. This is equivalent to the condition $v^{-1}(\chi)\in C_{\E}(\M,\iota\circ\ad(u))$. Thus the left square above also commutes. Therefore after passing to the isomorphism classes of objects, we conclude that the composition $\iota_\E\circ C_\E$ induces a canonical isomorphism of posets
\[C_S\colon\;\Mon^\circ_G(S)\xrightarrow{\sim}\Gamma(S,\SheafCone^W(\underline{X}^G_\QQ))\]
that is independent of the choice of the pinning $\E$ and compatible with the $\Aut_G(S)$-action on both sides. 

\section{The general classification theorem}\label{sec:general-classification}

\subsection{The case of a field}\label{sec:classification-field}
Let $G$ be a reductive group over a field $k$. Let $k_s$ be a separable closure of $k$ and let $\Gamma\coloneqq\mathrm{Gal}(k_s/k)$. We use notations from \S\ref{sec:absolute-RD}: so $\Psi$ is the absolute root datum of $G$ on which $\Gamma$ acts through the homomorphism $\rho_G$ in \eqref{eq:rho-G}. 
\begin{prop}\label{prop:absolute-characterization}
    Suppose that $k$ is a separably closed field. Then there exists a unique isomorphism of posets \[C=C_G\colon\;\Mon_G^\circ(k)\xrightarrow{\sim}\Cone^W(X_\QQ)\]
    such that for any maximal torus $T\subset G$, the following diagram of posets is commutative
    \[\xymatrix{
        \Mon_G^\circ(k)\ar[r]^{C_G\quad}_{\sim\quad}\ar[d] & \Cone^W(X_\QQ)\ar[d]\\
        \Mon_T^\circ(k)\ar[r]^{C_T\qquad}_{\sim\qquad} & \Cone_{\mathrm{fd}}(X^*(T)_\QQ)
    }\]
    where 
    \begin{itemize}
        \item the left vertical functor sends the isomorphism class of a pair $(\M,\iota)$ to the isomorphism class of $(\M_T,\iota|_T)$ where $\M_T$ is the schematic closure of $\iota(T)$ in $\M$;
        \item the right vertical functor is defined as follows: first choose a Borel $B$ containing $T$ so that the Killing pair $(T,B)$ determines an isomorphism $j_B\colon X\cong X^*(T)$ by Definition \ref{def:absolute-RD}, and induces the right vertical map; then we note that the $W$-orbit of $j_B$ is independent of the choice of $B$ so the right vertical map is also independent of the choice since the source consists of $W$-stable cones. 
    \end{itemize}
\end{prop}
\begin{proof}
    The existence of $C_G$ is a special case of the discussion in \S\ref{sec:main-thm-split-case}. The diagram commutes by the definition of $C_G$ from $C_\E$ for a pinning $\E$. The uniqueness follows from the injectivity of the right vertical map.   
\end{proof}

\begin{theo}
    Let $\Cone^W(X_\QQ)^{\rho_G(\Gamma)}$ be the sub-poset of Weyl cones $C\in\Cone^W(X_\QQ)$ such that $\rho_G(\sigma)(C)=C$ for all $\sigma\in\Gamma$. Then there exists a unique isomorphism of posets
    \[C\colon\;\Mon_G^\circ(k)\xrightarrow{\sim}\Cone^W(X_\QQ)^{\rho_G(\Gamma)}\]
    such that for any separably closed field $K$ containing $k$, it coincides with the restriction of the isomorphism in Proposition \ref{prop:absolute-characterization} for the split group $G_K$ after choosing an embedding $k_s\into K$ and identifying the absolute root data defined with respect to $k_s$ and $K$.
\end{theo}
\begin{proof}
    Let $(\M,\iota)$ be an object in $\Mon_G^\bullet(k)$, where $\M$ is a reductive monoid over $k$ and $\iota\colon G\xrightarrow{\sim}\M^\times$ is an isomorphism of $k$-groups. From Proposition \ref{prop:absolute-characterization} we get a Weyl cone $C(\M_{k_s})$ from the base change $(\M_{k_s},\iota)\in\Mon_G^\bullet(k_s)$. We first show that $C(\M_{k_s})$ is $\rho_G(\Gamma)$-invariant so that we get the desired functor.\par 
    Choose any Killing pair $(T,B)$ of $G_{k_s}$. Then for any $\sigma\in\Gamma$, the action of $\rho_G(\sigma)$ on $X_\QQ$ can be expressed as the composition of the following isomorphisms:
    \[X_\QQ\xrightarrow{\iota_{(T,B)}^{-1}} X^*(T)_\QQ\xrightarrow{(\sigma^{-1})^*}X^*(\sigma T)_\QQ\xrightarrow{\iota_{(T^\sigma,B^\sigma)}} X_\QQ\]
    where $\iota_{(T,B)}$ is the isomorphism from Lemma \ref{lem:Killing-pair-compatibility}. By the definition of $C(\M_{k_s})$ we have 
    \[\iota_{(T,B)}^{-1}(C(\M_{k_s}))=C(\overline{\iota(T)}),\quad \iota_{(T^\sigma,B^\sigma)}^{-1}(C(\M_{k_s}))=C(\overline{\iota(T^\sigma)}).\]
    On the other hand, the automorphism $\sigma$ on $\M_{k_s}$ sends $\overline{\iota(T)}$ to $\overline{\iota(T)^\sigma}$ and so
    \[(\sigma^{-1})^*C(\overline{\iota(T)})=C(\overline{\iota(T)^\sigma}).\]
    Since $\iota$ is defined over $k$, we have $\iota(T^\sigma)=\iota(T)^\sigma$ and taking closures in $\M_{k_s}$ we get $\overline{\iota(T^\sigma)}=\overline{\iota(T)^\sigma}$. Then from the above identities we deduce $\rho_G(\sigma)(C(\M_{k_s}))=C(\M_{k_s})$ and so $C(\M_{k_s})$ is $\rho_G(\Gamma)$-invariant. This shows that we get a well-defined functor $\Mon_G^\bullet(k)\to\Cone^W(X_\QQ)^{\rho_G(\Gamma)}$. It is faithful since both are thin categories. \par
    Next we show that the functor is full. Let $(\M_1,\iota_1)$ and $(\M_2,\iota_2)$ be two objects in $\Mon_G^\bullet(k)$ with cones $C_1=C(\M_{1,k_s})$ and $C_2=C(\M_{2,k_s})$. Suppose that $C_1\supset C_2$. By the equivalence in split case Theorem \ref{thm:split-classification}, there exists a monoid homomorphism $\M_{1,k_s}\to\M_{2,k_s}$ compatible with the rigidifications $\iota_1$ and $\iota_2$. Faithfulness in the split case forces this morphism to be invariant under Galois action and hence descends to a morphism $\M_1\to\M_2$ in $\Mon_G^\bullet(k)$ that corresponds to the inclusion $C_1\supset C_2$. \par     
    It remains to show essential surjectivity. Let $C\subset X_\QQ$ be a $\rho_G(\Gamma)$-stable Weyl cone. Let $G_0$ be a pinned split reductive group over $k$ with based root datum $\Psi$. Let $k'/k$ be a finite Galois extension contained in $k_s$ such that $G_{k'}$ is split and fix an isomorphism $G_{0,k'}\cong G_{k'}$. Then the homomorphism $\rho_G$ factors through a homomorphism $\mathrm{Gal}(k'/k)\to\Aut(\Psi)$.
    By Theorem \ref{thm:split-classification},  there exists a rigidified $G_{0,k'}$-monoid $(\M_0,\iota_0)\in\Mon_{G_{0,k'}}^\bullet(k')$ whose Weyl cone is $C$. The descent datum of $G$ is a cocycle of $\mathrm{Gal}(k'/k)$ with values in $\Aut_{G_0}(k')$. Since $C$ is $\rho_G(\Gamma)$-stable, the values of this cocycle belong to the subgroup $\Aut_{\M_0}(k')$ by Corollary \ref{cor:monoid-automorphism-field} and provide a descent datum for $(\M_0,\iota_0)$. Consequently we get an object $(\M_C,\iota)\in\Mon_G^\bullet(k)$ whose Weyl cone is $C$. Alternatively one can also construct the object $(\M_C,\iota)$ directly as follows. Recall from \cite[Theorem-Definition 2.17]{CJ} that the coordinate ring of $\M'$ is given by
    \[k'[\M']={}_{k'}\bfO(\L)\]
    where $\L=C\cap X^+$ is the weight monoid of $\M'$ and we refer to \emph{loc. cit.} for the notations and details. The upshot is that this expression, combined with the fact that $C$ and hence $\L$ are preserved by the $\rho_G(\Gamma)$ action, implies that the semilinear $\mathrm{Gal}(k'/k)$-action on $k'[G]$ preserves the subring $k'[\M']$. Then we define 
    \[\M_C\coloneqq\Spec \bigl((k'[\M'])^{\mathrm{Gal}(k'/k)}\bigr)\] 
    and the rigidification $\iota$ comes from the natural embedding 
    \[(k'[\M'])^{\mathrm{Gal}(k'/k)}\into (k'[G_{k'}])^{\mathrm{Gal}(k'/k)}=k[G].\]
\end{proof}

\subsection{The main theorem and applications}\label{sec:main-thm-general}    
Now we can state our main result.
\begin{theo}\label{thm:general-classification}
    Let $G$ be a reductive group scheme over a nonempty scheme $S$. Then the presheaf $\Mon_G^\circ$ on $S$ is representable by a scheme with an action of $\Aut_G$. Moreover, there exists a unique $\Aut_G$-equivariant isomorphism of schemes over $S$:
    \[C\colon\;\Mon_G^\circ\xrightarrow{\sim}\SheafCone^W(\underline{X}_\QQ^G)\]
    that satisfies the following properties:
    \begin{enumerate}
        \item For any $S$-scheme $S'$, it preserves the partial order on $S'$-points;
        \item On any geometric point $\Bar{s}\colon\Spec(k)\to S$ it gives the isomorphism $C_{G_{\Bar{s}}}$ in Proposition \ref{prop:absolute-characterization}.
    \end{enumerate}
    In particular, for any $S$-scheme $S'$ it induces a canonical bijection from $\Mon_G(S')$, the set of isomorphism classes of reductive $G_{S'}$-monoids, to the set of $\Aut_G(S')$-orbits on $\Hom_S(S',\SheafCone^W(\underline{X}_\QQ^G))$.
\end{theo}
\begin{proof}
    \emph{Step 1: Identify cone section with descent data.} By working on each constant type stratum $S_t$ separately we may and do assume that the type of $G$ is constant on $S$. Then there exists an \'etale covering $S'\to S$ (hence also an fpqc covering) such that $G_{S'}$ is split. Choose a pinning $\E$ of $G_{S'}$. We have constructed in \S\ref{sec:main-thm-split-case} a canonical isomorphism of posets:
    \[C_{S'}\colon\;\Mon^\circ_G(S')\xrightarrow{\sim}\Gamma(S',\SheafCone^W(\underline{X}^G_\QQ))\]
    For a section $c\in\Gamma(S,\SheafCone^W(\underline{X}^G_\QQ))$, let $c'\in \Gamma(S',\SheafCone^W(\underline{X}^G_\QQ))$ be its pullback and let $(\M',\iota')\in\Mon^\bullet_G(S')$ be an object corresponding to $c'$ under the above equivalence. Since the two pullbacks of $c'$ to $S'\times_SS'$ agree, by the full-faithfulness of the split classification in \S\ref{sec:main-thm-split-case} there exists a unique isomorphism between the two pullbacks of $(\M',\iota')$ to $S'\times_SS'$. Moreover, it satisfies the cocycle condition by the uniqueness of the isomorphism between pullbacks of $(\M',\iota')$ to the triple product $S'\times_SS'\times_SS'$. We conclude that $C_{S'}$ induces a bijection between $\Gamma(S,\SheafCone^W(\underline{X}^G_\QQ))$ and the set of isomorphism classes of descent data (with respect to $S'\to S$) for objects in $\Mon_G^\bullet(S')$.\par       
    \emph{Step 2: Bijection on $S$-points.} By Step 1, to show bijection on $S$-points it suffices to show that the natural functor from $\Mon_G^\bullet(S)$ to the category of descent data for objects in $\Mon_G^\bullet(S')$ with respect to the \'etale cover $S'\to S$ is an equivalence. In fact we will show the equivalence under the weaker assumption that $S'\to S$ is an fpqc cover. The full faithfulness of this functor follows from the fully-faithfulness of the descent data functor on the category of quasi-coherent sheaves, see \cite[\href{https://stacks.math.columbia.edu/tag/023T}{Proposition 023T}]{Stacks}. To show effectivity we first use the fact that fpqc descent data for affine schemes are effective, see \cite[\href{https://stacks.math.columbia.edu/tag/0244}{Section 0244}]{Stacks}. Once the underlying affine scheme descends, the monoid structure automatically descends. By \cite[\href{https://stacks.math.columbia.edu/tag/02L2}{Lemma 02L2}]{Stacks}, the flatness property descends. By \cite[\href{https://stacks.math.columbia.edu/tag/02KY}{Lemma 02KY}]{Stacks} and affineness, the finite presentation property descends. Finally the property of having integral normal geometric fibers is clearly fpqc local and hence also descends. Thus we have shown that given an object $(\M',\iota')\in\Mon_G^\bullet(S')$ with descent data relative to the fpqc covering $S'\to S$, we have obtained a reductive $G$-monoid $\M$ such that $\M_{S'}\cong\M'$. Finally, the descent data for the rigidification $\iota'\colon G_{S'}\xrightarrow{\sim}(\M')^\times$ is clearly also effective and induces a rigidification $\iota\colon G\xrightarrow{\sim}\M^\times$. This shows the effectivity of fpqc descent for objects in $\Mon^\bullet_G(S')$ and hence shows that the isomorphism $C_{S'}$ restricts to an isomorphism of posets:
    \[C_S\colon\Mon_G^\circ(S)\xrightarrow{\sim}\Gamma(S,\SheafCone^W(\underline{X}^G_\QQ)).\]
    Moreover the map $C_S$ is $\Aut_G(S)$-equivariant, as one can check on the splitting cover $S'$ and hence reduce to the split case in \S\ref{sec:main-thm-split-case}.\par 
    \emph{Step 3: Unique isomorphism of schemes.} We claim that the bijection in Step 2 is compatible with base change along any morphism $U\to S$. Indeed, after applying the descent argument on $U$ we are reduced to proving the claim in the split case. Then it follows from the construction in \S\ref{sec:main-thm-split-case}. Therefore we conclude that $\Mon^\circ_G$ is representable by a scheme over $S$ and there is an $\Aut_G$-equivariant isomorphism of $S$-scheme $C\colon\Mon^\circ_G\xrightarrow{\sim}\SheafCone^W(\underline{X}_\QQ^G)$. We already proved in Step 2 that condition (1) is satisfied. The base change compatibility shows that condition (2) is also satisfied. Finally since $\SheafCone^W(\underline{X}_\QQ^G)$ is an \'etale locally constant scheme over $S$, equality of two isomorphisms can be checked on geometric fibers and so the uniqueness of $C$ follows from Proposition \ref{prop:absolute-characterization}.    
\end{proof}
\begin{cor}\label{cor:Aut(M)}
    Let $(\M,\iota)$ be a rigidified reductive $G$-monoid corresponding to the section 
    \[c=C(\M,\iota)\in\Gamma(S,\SheafCone^W(\underline{X}_\QQ^G)).\]
    Then the group scheme of monoid automorphisms $\Aut_\M$ is identified, under the natural embedding $\Aut_\M\into\Aut_G$ induced by $\iota$, with the sub-group scheme $\Aut_{G,c}$ of stabilizers of $c$. 
\end{cor}

The reductive group $G$ itself is of course a reductive $G$-monoid, its corresponding Weyl cone is the whole weight space $\underline{X}^G_\QQ$. We say that a monoid scheme is \emph{genuine} if it is not a group scheme. The following result gives a simple criterion for the existence of genuine reductive $G$-monoids. 
\begin{cor}\label{cor:genuine-monoid}
    Let $G$ be a reductive group over a nonempty scheme $S$. Then a genuine reductive $G$-monoid exists if and only if $\Hom_{S\textrm{-gp}}(\GG_m,Z_G)\ne0$.
\end{cor}
\begin{proof}
    Suppose that there exists a nontrivial homomorphism of $S$-group schemes $\eta\colon\GG_m\to Z_G$. Then $\eta$ defines a nonzero global section of the sheaf of coweights $\underline{X}^{G,\vee}$. The condition $\eta\ge0$ defines a global section of $\SheafCone^W(\underline{X}^G_\QQ)$ different from the section of whole weight spaces and hence corresponds to a genuine reductive $G$-monoid by Theorem \ref{thm:general-classification}. More precisely, the global section defines a locally constant family of cones 
    \[C_{\eta,\Bar{s}}=\{\chi\in\underline{X}^G_{\QQ,\Bar{s}}\mid \langle\chi,\eta_{\Bar{s}}\rangle\ge0\},\qquad\forall\Bar{s}\in S.\]
    It is the whole space $\underline{X}^G_{\QQ,\Bar{s}}$ if $\eta_{\Bar{s}}=0$ and a half space otherwise. Since $\eta$ is not constantly zero by assumption, the corresponding rigidified $G$-monoid is genuine. \par 
    Conversely, let $(\M,\iota)$ be a genuine rigidified reductive $G$-monoid corresponding, under Theorem \ref{thm:general-classification}, to the section
    \[c=C(\M,\iota)\in\Gamma(S,\SheafCone^W(\underline{X}_\QQ^G)).\]
    Consider its image under the map \eqref{eq:omega_G}:
    \[\xi\coloneqq\omega_G(c)\in\Gamma(S,(\underline{X}^{G,\vee})^{W_G})=\Hom_{S\text{-gp}}(\GG_m,Z_G).\]
    Let $\Bar{s}\in S$ be a geometric point such that $\M_{\Bar{s}}$ is a genuine reductive $G_{\Bar{s}}$-monoid. Then the fiber $c_{\Bar{s}}$ of the cone section $c$ is not the whole space $\underline{X}^G_{\QQ,\Bar{s}}$. Moreover, the base change $\omega_{G,\Bar{s}}$ is given by the map $\omega$ in Definition \ref{def:Weyl-cone}, from which we see that $\xi_{\Bar{s}}\ne0$ and hence $\xi\ne0$.
\end{proof}

Let us mention another consequence of Theorem \ref{thm:general-classification}. It says that given a rigidified reductive monoid on the generic point of a normal integral scheme $S$, to extend it to a reductive monoid over $S$ it suffices to extend its unit group to a reductive group scheme over $S$. 

\begin{cor}\label{cor:integral-normal-base}
    Let $S$ be an integral normal scheme with generic point $\eta\in S$ and let $G$ be a reductive group scheme over $S$. Let $K=\kappa(\eta)$ be the residue field at $\eta$, i.e. the rational function field of $S$. Then taking generic fiber induces an equivalence of thin categories:
    \[\Mon_G^\bullet(S)\simeq\Mon^\bullet_{G_\eta}(K).\]
\end{cor}
\begin{proof}
    Choose an algebraic closure $\overline{K}$ of $K$ and let $\Bar{\eta}=\Spec(\overline{K})$ be the corresponding geometric point of $S$ above the generic point $\eta$. Let $K^{\mathrm{sep}}\subset\overline{K}$ be the separable closure of $K$. Let $X_\QQ$ be the fiber of $\underline{X}_\QQ^G$ at $\Bar{\eta}$. \par 
    Since $S$ is normal and integral, the \'etale sheaf $\underline{X}^G$ is isotrivial, i.e. splits over a finite \'etale cover of $S$ by \cite[proof of Corollary B.3.6]{ConradReductiveGroupScheme}. Hence the same holds for the sheaf $\SheafCone^W(\underline{X}_\QQ^G)$. In particular, it can be described by the \'etale fundamental group $\pi_1(S,\Bar{\eta})$ and we have
    \[\Gamma(S_{\mathrm{\acute{e}t}},\SheafCone^W(\underline{X}_\QQ^G))=\{C\in\Cone^W(X_\QQ)\mid \sigma(C)=C,\forall\sigma\in\pi_1(S,\Bar{\eta})\}\]
    By \cite[\href{https://stacks.math.columbia.edu/tag/0BQM}{Proposition 0BQM}]{Stacks} the canonical map 
    \[\mathrm{Gal}(K^{\mathrm{sep}}/K)=\pi_1(\eta,\Bar{\eta})\to\pi_1(S,\Bar{\eta})\]
    is surjective. Therefore the restriction map to the generic fiber is an isomorphism of posets
    \[\Gamma(S_{\mathrm{\acute{e}t}},\SheafCone^W(\underline{X}_\QQ^G))\xrightarrow{\simeq}\Gamma(\eta_{\mathrm{\acute{e}t}},\SheafCone^W(\underline{X}_\QQ^G))\]
    and both are identified with the poset $\Cone^W(X_\QQ)^{\mathrm{Gal}(K^{\mathrm{sep}}/K)}$. Then we conclude by Theorem \ref{thm:general-classification}. 
\end{proof}

\section{Some Basic Examples}\label{sec:example}
\subsection{Central simple algebras}
Let $n$ be a positive integer and let $A$ be a central simple algebra of dimension $n^2$ over a field $k$. Its unit group $G=A^\times$ is an inner form of $G_0=\GL_n$ over $k$ and its multiplicative monoid $\M_A$ is a reductive $G$-monoid over $k$. The abstract Cartan $T$, the absolute based root datum $(X,X^\vee,\Delta,\Delta^\vee)$, and the abstract Weyl group $W$ of $G$ are isomorphic to those of its split inner form $\GL_n$. So we have $T=(\GG_m)^n$, $X=\ZZ^n$, $\Delta=\{\alpha_i=e_i-e_{i+1},1\le i\le n-1\}$ and $W=S_n$ acting by permuting coordinates on $X$ etc. Here the $e_i$ (for $1\le i\le n$) are the standard basis elements for $X=\ZZ^n$. The Weyl cone for $\M_A$ is 
\[C(\M_A)=\bigoplus_{1\le i\le n}\QQ_{\ge0}\cdot e_i=(\QQ_{\ge 0})^n\subset\QQ^n.\]
The outer automorphism group is the constant group scheme $\Out_G=\Out_{G_0}=\ZZ/2\ZZ$. The nontrivial outer automorphism $\sigma$ acts on $X=\ZZ^n$ by $(x_1,\dotsc,x_n)\mapsto(-x_n,-x_{n-1},\dotsc,-x_1)$ and hence $\sigma(\alpha_i)=\alpha_{n-i}$ for all $1\le i\le n-1$. In particular, we have $\sigma(C(\M_A))=-C(\M_A)$. According to \cite[Example 7.1.12]{ConradReductiveGroupScheme}\footnote{Note that in \emph{loc. cit.} the groups are inner forms of $\mathrm{SL}_n$, but the same arguments apply in our situation of inner forms of $\GL_n$, with the extra bonus that we do not need to exclude the $n=2$ case: although $\Out_{\mathrm{SL}_2}=1$ we still have $\Out_{\mathrm{GL}_2}=\ZZ/2\ZZ$.}, the natural map $\Aut_G(k)\to\Out_G(k)=\ZZ/2\ZZ$ is surjective if and only if the class of $A$ in the Brauer group $\mathrm{Br}(k)$ is $2$-torsion, i.e. $A\cong A^{op}$. So in these cases, the cones $C(\M_A)$ and $-C(\M_A)$ form a single $\Aut_G(k)$-orbit and correspond to the isomorphism class of the monoid $\M_A$. Two cones in the orbit correspond to distinct isomorphism classes of rigidified reductive $G$-monoids whose underlying monoids are isomorphic. On the other hand, if the class of $A$ is not $2$-torsion, equivalently the algebras $A$ and $A^{op}$ are not isomorphic, then the $\Aut_G(k)$-orbit of $C(\M_A)$ consists of itself while the opposite cone $-C(\M_\A)$ corresponds to $\M_{A^{op}}$, the multiplicative monoid of $A^{op}$, together with the rigidification $A^\times\to (A^{op})^\times$, $a\mapsto a^{-1}$.

\subsection{The Vinberg monoid}\label{sec:Vinberg-monoid}
Let $G$ be a reductive group scheme over a nonempty scheme $S$, with abstract Cartan $T_G$. It comes with an embedding $Z_G\into T_G$ and a canonical isomorphism $T_G/Z_G\cong T_{G_\ad}$, the abstract Cartan for the adjoint group $G_\ad$ (see Definition \ref{def:abstract-Cartan}). Define the \emph{enhanced group of $G$} to be the following reductive group scheme over $S$:
\[G_\enh\coloneqq (T_G\times G)/Z_G\]
with abstract Cartan
\[T_{G_\enh}=(T_G\times T_G)/Z_G\]
where the quotients are by the anti-diagonal action of $Z_G$. Note that the center of $G_{\enh}$ is isomorphic to the abstract Cartan of $G$:
\[Z_{G_\enh}=(T_G\times Z_G)/Z_G\cong T_G.\]
We will define the universal enveloping monoid (or Vinberg monoid) for $G$, which is a reductive $G_\enh$-monoid, generalizing the construction of \cite[\S5.2, \S5.3]{CJ} in the split case. For this we may assume that $G$ (and hence $G_\enh$) is of constant type.\par 
Let $G_0$ be the split form of $G$. The sheaf of based root data for $G_0$ is constant with fiber $(X,X^\vee,\Delta,\Delta^\vee)$ at a chosen geometric point $\Bar{s}\in S$. Recall that the sheaf of based root data for $G$ is described by
\[(\underline{X}^G,\underline{X}^{G,\vee},\underline{\Delta}^G,\underline{\Delta}^{G,\vee})=\SheafIsom_S^{\Out}(G_0,G)\times^{\Out_{G_0}}(X,X^\vee,\Delta,\Delta^\vee).\]
Next we consider similar objects defined for the enhanced group $G_\enh$. The split form of $G_\enh$, together with its abstract Cartan, are given by 
\[G_{0,\enh}=(T_{G_0}\times G_0)/Z_{G_0},\qquad T_{G_{0,\enh}}=(T_{G_0}\times T_{G_0})/Z_{G_0}\]
where $T_{G_0}$ is the abstract Cartan of $G_0$ and $Z_{G_0}$ embeds anti-diagonally.\par 
The sheaf of based root data for $G_{0,\enh}$ is constant and its fiber at $\Bar{s}$ is $(X_{\enh},X_{\enh}^\vee,\Delta_\enh,\Delta_\enh^\vee)$ where
\[X_\enh\coloneqq\{(\mu,\la)\in X\oplus X\mid \mu|_{Z_{G_0}}=\la|_{Z_{G_0}}\},\quad\Delta_\enh=\{(0,\alpha)\in X_\enh\mid \alpha\in\Delta\}.\]
The abstract Weyl group of $G_{0,\enh}$ coincides with that of $G_0$ and we let $W$ be its fiber at $\Bar{s}\in S$. Then $W$ acts on $X_\enh$ through the second factor. We let $\Out_{G_0}$ act \emph{diagonally} on the two factors of $X_\enh$. Then the sheaf of based root data for $G_\enh$ is the \'etale locally constant sheaf on $S$ described by 
\begin{align*}
    &(\underline{X}^{G_\enh},\underline{X}^{G_\enh,\vee},\underline{\Delta}^{G_\enh},\underline{\Delta}^{G_\enh,\vee})\\ &=\SheafIsom_S^{\Out}(G_0,G)\times^{\Out_{G_0}}(X_\enh,X_\enh^\vee,\Delta_\enh,\Delta_\enh^\vee).
\end{align*}
Let $\QQ_{\ge0}[\Delta]\subset X_\QQ$ be the cone generated by the simple roots and consider the following convex polyhedral cone in $X_{\enh,\QQ}$:
\[C_{\Vin}\coloneqq \left\{(\mu,\la)\in X_{\enh,\QQ}\subset X_\QQ\oplus X_\QQ\mid \mu-w(\la)\in\QQ_{\ge0}[\Delta]\text{ for all }w\in W\right\}.\]
The definition exhibits $C_\Vin$ as the intersection of finitely many half spaces bounded by rational hyperplanes and hence it is a rational convex polyhedral cone in $X_{\enh,\QQ}$. It is also preserved by the $W$-action (which acts on the second factor). Moreover, the projection to the second $X_\QQ$-factor induces a surjection $C_\Vin\epic X_\QQ$ whose fibers contain some translates of the cone $\QQ_{\ge0}[\Delta]$. This shows that $C_\Vin$ spans $X_{\enh,\QQ}$ as a $\QQ$-vector space and we conclude that $C_\Vin\in\Cone^W(X_{\enh,\QQ})$. Moreover it is stable under the diagonal $\Out_{G_0}$ action and therefore gives rise to a global section
\[C_{\Vin}\in\Gamma(S,\SheafCone^W(\underline{X}^{G_\enh}_\QQ)).\]
The \emph{Vinberg monoid} for $G$, denoted by $\Vin_G$, is defined to be the reductive $G_\enh$-monoid corresponding to the $\Aut_{G_\enh}(S)$-orbit of $C_\Vin$ under the bijection in Theorem \ref{thm:general-classification}. The cone $C_\Vin$ also gives rise to a rigidification $\iota\colon G_\enh\cong(\Vin_G)^\times$. 
\begin{ex}
    Let $G=\mathrm{SU}_{2,\RR}$ be the special unitary group associated to the positive definite Hermitian form on $\mathbb{C}^2$ of signature $(2,0)$. Equivalently $G$ can also be defined as the norm-one subgroup of the multiplicative group of the Hamiltonian quaternion algebra $\mathbb{H}$. It is an inner form of $G_0=\mathrm{SL}_{2,\RR}$ and so the abstract Cartan is $T_G=T_{G_0}=\GG_m$. The enhanced group 
    \[G_\enh=(\GG_{m,\RR}\times\mathrm{SU}_{2,\RR})/\mu_2\] 
    is also the unit group of $\mathbb{H}$, and it is an inner form of $G_{0,\enh}=\GL_{2,\RR}$. The Vinberg monoid for $G_0$ is $\mathrm{Mat}_{2,\RR}$ while the Vinberg monoid for $G$ is the multiplicative monoid of $\mathbb{H}$. 
\end{ex}

%\bibliographystyle{alpha}
%\bibliography{Vinberg}
\end{document}